\documentclass[reqno,12pt]{amsart}

\usepackage{amsmath, epsfig, cite}
\usepackage{amssymb}
\usepackage{amsfonts}
\usepackage{latexsym}
\usepackage{amsthm}

\newtheorem{theorem}{Theorem}[section]

\newtheorem{corollary}[theorem]{Corollary}

\theoremstyle{remark}

\numberwithin{equation}{section}

\allowdisplaybreaks[1]
\title[Proof of a $q$-supercongruence conjectured by Guo and Schlosser]
{Proof of a $q$-supercongruence conjectured by Guo and Schlosser}

\author{Long Li}
\address{School of Mathematics and Statistics, Huaiyin Normal University,
Huai'an 223300, Jiangsu, People's Republic of China}
\email{lli@hytc.edu.cn}
\thanks{The first author was partially supported by the Natural Science Foundation of the Jiangsu Higher Education Institutions of China (grant 19KJB110006).}
\author{Su-Dan Wang$^*$}
\address{College of Mathematics Science, Inner Mongolia Normal University, Huhhot 010022, Inner Mongolia}
\email{sdwang@imnu.edu.cn}
\thanks{*Corresponding author.}

\subjclass[2010]{Primary 11B65; Secondary 11A07, 11F33}

\keywords{cyclotomic polynomial, $q$-binomial coefficients, supercongruences, identities}

\begin{document}

\date{\today}

\begin{abstract}
In this paper, we confirm the following conjecture of Guo and Schlosser: for any odd integer $n>1$ and $M=(n+1)/2$ or $n-1$,
$$
\sum_{k=0}^{M}[4k-1]_{q^2}[4k-1]^2\frac{(q^{-2};q^4)_k^4}{(q^4;q^4)_k^4}q^{4k}\equiv (2q+2q^{-1}-1)[n]_{q^2}^4\pmod{[n]_{q^2}^4\Phi_n(q^2)},
$$
where $[n]=[n]_q=(1-q^n)/(1-q),(a;q)_0=1,(a;q)_k=(1-a)(1-aq)\cdots(1-aq^{k-1})$ for $k\geq 1$ and $\Phi_n(q)$
denotes the $n$-th cyclotomic polynomial.
\end{abstract}

\maketitle

\section{Introduction}

\noindent In 1997, Van Hamme \cite{Hamme} observed that $13$ supercongruences on truncated forms of Ramanujan's and Ramanujan-like formulas for $1/\pi$.
In particular, the following supercongruence of Van Hamme \cite[(B.2)]{Hamme},
\begin{equation*}
\sum_{k=0}^{(p-1)/2}\frac{4k+1}{(-64)^k}{2k\choose k}^3\equiv p(-1)^{\frac{p-1}{2}}\pmod {p^3}
\end{equation*}
was first proved by Mortenson \cite{Mortenson} using a $_6F_5$ transformation and a technical evaluation of a quotient of Gamma functions,
where $p$ is an odd prime. Recently, $q$-analogues of congruences and supercongruences have caught the interests of many authors
(see, for example, \cite{Guillera,Guo-m3,Guo-a2,Guo-jmaa,Guo-Zudlin,Guo-result,GS3,GW,Liu-Petrov,Liu-Huang,NP,
Tauraso,WY,Zudilin}).
In \cite[Conjecture 4.3]{Guo-ITSF}, Guo conjectured: for any prime $p>3$ and positive integer $r$,
\begin{align}
\sum_{k=0}^{(p^r-1)/2}\frac{(4k+1)^3}{256^k}{2k\choose k}^4 &\equiv -p^r \pmod {p^{r+3}}, \label{eq1.11}\\
\sum_{k=0}^{p^r-1}\frac{(4k+1)^3}{256^k}{2k\choose k}^4     &\equiv -p^r \pmod {p^{r+3}}. \label{eq1.12}
\end{align}
Later, Guo \cite[Theorem 1.1]{Guo-result} proved \eqref{eq1.11} and \eqref{eq1.12} by establishing the following complete $q$-analogues of them:
for odd integer $n>1$, modulo $[n]_{q^2}\Phi_n(q^2)^3$,
\begin{align*}
\sum_{k=0}^{(n-1)/2}[4k+1]_{q^2}[4k+1]^2\frac{(q^2;q^4)_k^4}{(q^4;q^4)_k^4}q^{-4k}
&\equiv -[n]_{q^2}\frac{2q^{2-n}}{1+q^2}-[n]_{q^2}^3\frac{(n^2-1)(1-q^2)^2q^{2-n}}{12(1+q^2)},  
\\
\sum_{k=0}^{n-1}[4k+1]_{q^2}[4k+1]^2\frac{(q^2;q^4)_k^4}{(q^4;q^4)_k^4}q^{-4k}
&\equiv
-[n]_{q^2}\frac{2q^{2-n}}{1+q^2}-[n]_{q^2}^3\frac{(n^2-1)(1-q^2)^2q^{2-n}}{12(1+q^2)}.  
\end{align*}
Here and throughout the paper, $(a;q)=(1-a)(1-aq)\cdots(1-aq^{n-1})$ denotes the $q$-\emph{shifted factorial},
$[n]=[n]_q=1+q+\cdots+q^{n-1}$ stands for the \emph{$q$-integer}, and $\Phi_n(q)$ is
the $n$-th \emph{cyclotomic polynomial} in $q$, i.e.,
$$\Phi_n(q)=\prod_{\substack{1\leq k\leq n\\\gcd(n,k)=1}}^n(q-\zeta^k)$$
with $\zeta$ being an $n$-th primitive root of unity.

Guo \cite[Theorem 1.2]{Guo-result} also proved that, for any odd integer $n>1$,
\begin{align}
\sum_{k=0}^{(n+1)/2}[4k-1]_{q^2}[4k-1]^2\frac{(q^{-2};q^4)_k^4}{(q^4;q^4)_k^4}q^{4k} &\equiv 0\pmod{[n]_{q^2}\Phi_n(q^2)^3},  \label{eq1.17} \\
\sum_{k=0}^{n-1}[4k-1]_{q^2}[4k-1]^2\frac{(q^{-2};q^4)_k^4}{(q^4;q^4)_k^4}q^{4k}     &\equiv 0\pmod{[n]_{q^2}\Phi_n(q^2)^3},  \label{eq1.18}
\end{align}
thus confirming the $m=3$ case \cite[Conjecture 5.2]{GL}.

The aim of this paper is to prove the following refinements of \eqref{eq1.17} and \eqref{eq1.18}, which were originally conjectured by
Guo and Schlosser \cite[Conjecture 3]{GS3} (The modulus $[n]_{q^2}^4$ case was first formulated by Guo \cite[Conjecture 6.3]{Guo-result}).
\begin{theorem}\label{thm3} Let $n>1$ be an odd integer. Then
\begin{align}
\sum_{k=0}^{(n+1)/2}[4k-1]_{q^2}[4k-1]^2\frac{(q^{-2};q^4)_k^4}{(q^4;q^4)_k^4}q^{4k} &\equiv (2q+2q^{-1}-1)[n]_{q^2}^4\pmod{[n]_{q^2}^4\Phi_n(q^2)},
\label{eq1.3} \\
\sum_{k=0}^{n-1}[4k-1]_{q^2}[4k-1]^2\frac{(q^{-2};q^4)_k^4}{(q^4;q^4)_k^4}q^{4k}     &\equiv (2q+2q^{-1}-1)[n]_{q^2}^4\pmod{[n]_{q^2}^4\Phi_n(q^2)}. \label{eq1.4}
\end{align}
\end{theorem}

Let  $n=p^r$ be an odd prime power, and take $q\to 1$ in \eqref{eq1.3} and \eqref{eq1.4}. Noticing that
$$
\lim_{q\to 1}\frac{(q^{-2},q^4)_k}{(q^4;q^4)_k}=\frac{-1}{4^k(2k-1)}{2k\choose k},
$$
we immediately obtain the following conclusion, which was observed by Guo \cite[Conjecture 6.4]{Guo-m3}.
\begin{corollary} Let $p$ be an odd prime and $r$ a positive integer. Then
$$
\sum_{k=0}^{(p^r+1)/2}\frac{(4k-1)^3}{256^k(2k-1)^4}{2k\choose k}^4\equiv 3p^{4r} \pmod{p^{4r+1}},
$$
$$
\sum_{k=0}^{p^r-1}\frac{(4k-1)^3}{256^k(2k-1)^4}{2k\choose k}^4\equiv 3p^{4r} \pmod{p^{4r+1}}.
$$
\end{corollary}

The remainder of this paper is organized as follows. In the next section, we show the proof of Theorem \ref{thm3}. Section 3 contains a parameter-generalization of Theorem \ref{thm2} and a proof of \cite[Theorem 4.3]{Guo-result} modulo $[n]_{q^2}^2(1-aq^{2n})(a-q^{2n})$.

\section{Proof of Theorem \ref{thm3}}

We first prove the following identity, which plays an important role in our proof of Theorem \ref{thm3}.

\begin{theorem}\label{thm2} Let $n$ be a positive integer. Then
 \begin{equation}\label{eq:thm2.1}
 \begin{split}
 &\sum_{k=0}^{n-1}[4k-1]_{q^2}[4k-1]^2\frac{(q^{-2};q^4)_k^4}{(q^4;q^4)_k^4}q^{4k}\\
  &\quad=(q^{2n}+1)^4[n]_{q^2}^4\frac{(q^{-2};q^4)^4_{n}}{(q^4;q^4)_{n}^4}
  \biggl(2\cdot\frac{q^5+q^{4n+1}(q^{4n-2}-q^2-1)}{(q^2-1)^2}-q^{4n}\biggr).
  \end{split}
  \end{equation}
\end{theorem}

Note that \eqref{eq:thm2.1} could be regarded as a $q$-analogue of the following identity:
$$
\sum_{k=0}^{n-1}\frac{(4k-1)^3}{256^k(2k-1)^4}{2k\choose k}^4=\frac{16n^4(8n^2-12n+3){2n\choose n}^4}{256^n(2n-1)^4}.
$$

\begin{proof}
For convenience, let
$$
S(n)=\sum_{k=0}^{n-1}[4k-1]_{q^2}[4k-1]^2\frac{(q^{-2};q^4)_k^4}{(q^4;q^4)_k^4}q^{4k},
$$ and
$$
T(n)=(q^{2n}+1)^4[n]_{q^2}^4\frac{(q^{-2};q^4)^4_{n}}{(q^4;q^4)_{n}^4}f_n(q),
$$
where
$$
f_n(q)=2\cdot\frac{q^5+q^{4n+1}(q^{4n-2}-q^2-1)}{(q^2-1)^2}-q^{4n}.
$$
We proceed by induction on $n$. For $n=1,$ it is clear that
$$
S(1)=-\frac1{q^4}=T(1).
$$
Suppose that the statement is true for $n$. We now consider the $n+1$ case.
Using the induction hypothesis, we have
\begin{align*}
S(n+1)&=S(n)+[4n-1]_{q^2}[4n-1]^2\cdot\frac{(q^{-2};q^4)_n^4}{(q^4;q^4)_n^4}q^{4n}\\[5pt]
&=\frac{(q^{-2};q^4)_n^4}{(q^4;q^4)_n^4}\biggl([2n]_{q^2}^4f_n(q)+[4n-1]_{q^2}[4n-1]^2q^{4n} \biggr)\\[5pt]
&=\frac{(q^{-2};q^4)_{n+1}^4(1-q^{4n+4})^4}{(q^4;q^4)_{n+1}^4(1-q^{4n-2})^4}\biggl([2n]_{q^2}^4f_n(q)
+[4n-1]_{q^2}[4n-1]^2q^{4n}\biggr)\\[5pt]
&=(1+q^{2n+2})^4[n+1]_{q^2}^4\frac{(q^{-2};q^4)_{n+1}^4}{(q^4;q^4)_{n+1}^4(1-q^{4n-2})^4}\\[5pt]
&\quad\times\biggl((1-q^{4n})^4f_n(q)+(1-q^{2(4n-1)})(1-q^{4n-1})^2(1-q)(1+q)^3q^{4n}\biggr)\\[5pt]
&=T(n+1),
\end{align*}
where the last equality holds because of the following relation
\begin{align*}
&(1-q^{4n})^4f_n(q)+(1-q^{2(4n-1)})(1-q^{4n-1})^2(1-q)(1+q)^3q^{4n}\\
&\quad=(1-q^{4n-2})^4f_{n+1}(q).
\end{align*}
This completes the proof of Theorem \ref{thm2}.
\end{proof}

We are now able to prove Theorem \ref{thm3}.

\begin{proof}[Proof of \eqref{eq1.3}]
Replacing $n$ by $(n+3)/2$ in \eqref{eq:thm2.1}, we obtain
\begin{align}
\sum_{k=0}^{(n+1)/2}[4k-1]_{q^2}[4k-1]^2\frac{(q^{-2};q^4)_k^4}{(q^4;q^4)_k^4}q^{4k}
&=(q^{n+3}+1)^4\frac{(1-q^{n+3})^4(q^{-2};q^4)_{(n+3)/2}^4}{(1-q^2)^4(q^4;q^4)_{(n+3)/2}^4}
f_{\frac{n+3}{2}}(q)  \notag\\
&=[n]_{q^2}^4\frac{(q^{-2};q^4)_{(n+1)/2}^4}{(q^4;q^4)_{(n+1)/2}^4}f_{\frac{n+3}2}(q).  \label{eq3.4}
\end{align}
In light of \cite[(4.2)]{Guo-result}, we have
$$
\frac{(q^{-2};q^4)_{(n+1)/2}}{(q^4;q^4)_{(n+1)/2}}\equiv(-1)^{(n+1)/2}q^{(n-1)^2/2-2}\pmod{\Phi_n(q^2)},
$$
and so the right-hand side of \eqref{eq3.4} is congruent to
$$
[n]_{q^2}^4q^{2(n-1)^2-8}\cdot q^5(2q^2-q+2)\equiv [n]_{q^2}^4(2q+2q^{-1}-1)
$$
modulo $[n]_{q^2}^4\Phi(q^2)$. This completes the proof.
\end{proof}

\begin{proof}[Proof of \eqref{eq1.4}]
It is easy to see that
$$
\frac{(q;q^2)_n}{(q^2;q^2)_n}=\frac 1{(-q;q)^2_n}{2n\brack n},
$$
where the \emph{$q$-binomial coefficients} ${n\brack k}$ are defined by
$$
{n\brack k}={n\brack k}_q
=\begin{cases}
\dfrac{(q;q)_n}{(q;q)_k(q;q)_{n-k}} & \mbox{if } 0\leq k\leq n, \\[15pt]
0 & \mbox{otherwise}.
\end{cases}$$
So the identity \eqref{eq:thm2.1} may be restated as
\begin{align}
&\sum_{k=0}^{n-1}[4k-1]_{q^2}[4k-1]^2\frac{(q^{-2};q^4)_k^4}{(q^4;q^4)_k^4}q^{4k} \notag\\
&\quad=\frac{(q^{2n}+1)^4[n]_{q^2}^4(q^2-1)^4}{(q^2-q^{4n})^4(-q^2;q^2)_n^8}{2n\brack n}_{q^2}^4\left(2\cdot\frac{q^5+q^{4n+1}(q^{4n-2}-q^2-1)}{(q^2-1)^2}-q^{4n}\right).
\label{eq3.1}
\end{align}
Since $q^n\equiv 1\pmod{\Phi_n(q)}$,  by \cite[Lemma 3.1]{GW} we have
\begin{equation}\label{eq3.2}
{2n\brack n}_{q^2}=(1+q^{2n}){2n-1\brack n-1}_{q^2}\equiv 2(-1)^{n-1}q^{n(n-1)}\equiv 2\pmod{\Phi_n(q^2)}
\end{equation}
for odd $n>1$. In view of \cite[Lemma 3.2]{GW}, the following $q$-congruence holds:
\begin{equation}\label{eq3.3}
(-q^2;q^2)_n=(1+q^{2n})(-q^2;q^2)_{n-1}\equiv 2\pmod {\Phi_n(q^2)}.
\end{equation}
Applying \eqref{eq3.2} and \eqref{eq3.3}, we see that the right-hand side of \eqref{eq3.1} is congruent to
$$
[n]_{q^2}^4\left(2\cdot\frac{q^5+q(q^{-2}-q^2-1)}{(q^2-1)^2}-1\right)=[n]_{q^2}^4(2q+2q^{-1}-1)
$$
modulo $[n]_{q^2}^4\Phi(q^2)$.  This completes the proof.
\end{proof}

\section{A further generalization of Theorem \ref{thm2}}

By induction on $n$, we can also prove the following parameter generalization of \eqref{eq:thm2.1}.

\noindent \begin{theorem}\label{thm4} Let $n>1$ be an integer. Then
\begin{align}
&\sum_{k=0}^{n-1}[4k-1]_{q^2}[4k-1]^2\frac{(q^{-2};q^4)_k^2(q^{-2}/a;q^4)_k(aq^{-2};q^4)_k}
{(q^4;q^4)_k^2(q^4/a;q^4)_k(aq^4;q^4)_k}q^{4k}  \notag\\
&\quad=q(q^{2n}+1)^2[n]_{q^2}^2{2n\brack n}_{q^2}^2
\frac{(1-q^{-2})^2(q^6/a;q^4)_{n-2}(aq^6;q^4)_{n-2}f_n(a,q)}
{(1-q^{4n-2})^2(-q^2;q^2)_n^4(aq^4;q^4)_{n-1}(q^4/a;q^4)_{n-1}},  \label{eq:thm4}
\end{align}
where
\begin{align*}
f_n(a,q)&=2+2\sum_{i=1}^{4n-6}i(q^i+q^{8n-6-i})
-\frac{(a^2+1)(q^{8n-5}-2q^{4n-1}+2q^{4n-2}-2q^{4n-3}+q)}{a(q-1)^2}
\\&\quad+(8n-10)q^{4n-4}(1+q^2)+(8n-11)q^{4n-5}(1+q^4)+(8n-8)q^{4n-3}+2q^{8n-6}.
\end{align*}
\end{theorem}

We end this paper with the following generalization of \cite[Theorem 4.3]{Guo-result}.
\begin{theorem}\label{thm5} Let $n>1$ be an odd integer and $a$ an indeterminate. Then, modulo $[n]_{q^2}^2(1-aq^{2n})(a-q^{2n})$,
\begin{equation}\label{eq:thm5}
\sum_{k=0}^{M}[4k-1]_{q^2}[4k-1]^2\frac{(q^{-2};q^4)_k^2(q^{-2}/a;q^4)_k(aq^{-2};q^4)_k}
{(q^4;q^4)_k^2(q^4/a;q^4)_k(aq^4;q^4)_k}q^{4k}\equiv 0,
\end{equation}
 where $M=(n+1)/2$ or $n-1$.
\end{theorem}

\begin{proof}
Let $n\mapsto (n+3)/2$ in \eqref{eq:thm4}, and notice the following realtion
\begin{align*}
\frac{1}{(1-q^{2n+4})^2}\biggl[\frac{n+3}2\biggr]_{q^2}^2{n+3\brack \frac{n+3}{2}}_{q^2}^2
=[n]_{q^2}^2{n-1\brack\frac{n-1}{2} }_{q^2}^2(1+q^{n+3})^2\frac{(1+q^{n+1})^2}{(1-q^{n+1})^2}.
\end{align*}
Since
\begin{equation*}
\gcd(1-q^n,1+q^m)=1\quad\text{and}\quad\gcd([n],[2n-1])=1
\end{equation*}
for all positive integers $m$ and $n$ with $n$ odd, we conclude that
the $q$-congruence \eqref{eq:thm5} holds for $M=n-1$ and $(n+1)/2$.





\end{proof}


\begin{thebibliography}{10}
\bibitem{Guillera} J. Guillera, WZ pairs and $q$-analogues of Ramanujan series for $1/\pi$, J. Difference Equ. Appl. 24 (2018), 1871--1879.

\bibitem{Guo-ITSF} V.J.W. Guo, Some generalizations of a supercongruence of van Hamme, Integral Transforms Spec. Funct. 28 (2017), 888--899.

\bibitem{Guo-m3}V.J.W. Guo, Common $q$-analogues of some different supercongruences,
Results Math. 74 (2019), Art. 131.

\bibitem{Guo-result} V.J.W. Guo, Proof of some $q$-supercongruences modulo the fourth power of a cyclotomic polynomial, Results Math. 75 (2020), Art. 77.

\bibitem{Guo-a2}V.J.W. Guo, A $q$-analogue of the (A.2) supercongruence of Van Hamme for primes $p\equiv 1\pmod 4$,
Rev. R. Acad. Cienc. Exactas F\'is. Nat. Ser. A Mat. RACSAM 114 (2020), Art.~123.

\bibitem{Guo-jmaa}V.J.W. Guo, $q$-Analogues of Dwork-type supercongruences, J. Math. Anal. Appl. 487 (2020), Art.~124022.

\bibitem{GL} V.J.W. Guo and J.-C. Liu, Some congruences related to a congruence of Van Hamme, Integral Transforms Spec. Funct. 31 (2020), 221--231.

\bibitem{GS3} V.J.W. Guo and M.J. Schlosser, A family of $q$-hypergeometric congruences modulo the fourth power of a cyclotomic polynomial, Israel J. Math., to appear.

\bibitem{GW} V.J.W. Guo and S.-D. Wang, Some congruences involving fourth powers of central $q$-binomial coefficients, Proc. Roy. Soc. Edinburgh Sect. A 150 (2020), 1127--1138.

\bibitem{Guo-Zudlin} V.J.W. Guo and W. Zudilin, A $q$-microscope for supercongruences, Adv. Math. 346 (2019), 329--358.

\bibitem{Liu-Huang} J.-C. Liu and Z.-Y. Huang, A truncated identity of Euler and related $q$-congruences, Bull. Aust. Math. Soc., https://doi:10.1017/S0004972720000301

\bibitem{Liu-Petrov} J.-C. Liu and F. Petrov, Congruences on sums of $q$-binomial coefficients, Adv. Appl. Math. 116 (2020), Art.~102003.


\bibitem{Long} L. Long, Hypergeometric evaluation identities and supercongruences, Pacific J. Math. 249 (2011), 405--418.

\bibitem{Mortenson} E. Mortenson, A $p$-adic supercongruence conjecture of van Hamme, Proc. Amer. Math. Soc. 136 (2008), 4321--4328.

\bibitem{NP}H.-X. Ni, H. Pan, On a conjectured $q$-congruence of Guo and Zeng, Int. J. Number Theory 14 (2018), 1699--1707.

\bibitem{Tauraso} R. Tauraso, Some $q$-analogs of congruences for central binomial sums, Colloq. Math. 133 (2013), 133--143.

\bibitem{Hamme} L. Van Hamme, Some conjectures concerning partial sums of generalized hypergeometric series,
in: $p$-Adic Functional Analysis (Nijmegen, 1996), Lecture Notes in Pure and Appl. Math. 192, Dekker, New York, 1997, pp. 223--236.

\bibitem{WY}X. Wang and M. Yue, Some $q$-supercongruences from Watson's $_8\phi_7$ transformation formula,
Results Math. 75 (2020), Art. 71.

\bibitem{Zudilin}W. Zudilin, Congruences for $q$-binomial coefficients, Ann. Combin. 23 (2019), 1123--1135.

\end{thebibliography}
\end{document}